\documentclass[12pt, a4paper, leqno, twoside]{amsart}
\usepackage{amsmath, amsthm, amsfonts, amssymb, graphicx, color, cancel, ulem}
\usepackage{esint}
\numberwithin{equation}{section}
\newtheorem{thm}{Theorem}[section]
\newtheorem{cor}[thm]{Corollary}
\newtheorem{prop}[thm]{Proposition}
\newtheorem{lem}[thm]{Lemma}
\newtheorem{defn}[thm]{Definition}
\newtheorem{rem}[thm]{Remark}

\newtheorem*{remark*}{Remark}

\newcommand{\N}{\mathbb{N}}
\newcommand{\R}{\mathbb{R}}

\newcommand{\cG}{\mathcal{G}}

\newcommand{\cM}{\mathcal{M}}

\newcommand{\comp}{\mathrm{C}}
\newcommand{\weak}{\mathrm{weak}}

\newcommand{\ls}{\lesssim}
\newcommand{\gs}{\gtrsim}

\newcommand{\vp}{\varphi}
\newcommand{\ve}{\varepsilon}

\newcommand{\cGdec}{\mathcal{G}^{\rm dec}}

\newcommand{\Cic}{C^{\infty}_{\comp}}
\newcommand{\Li}{L^{\infty}}
\newcommand{\Lic}{L^{\infty}_{\comp}}

\newcommand{\iPy}{{\it{\Phi_Y}}}

\newcommand{\dlim}{\displaystyle\lim}
\newcommand{\dint}{\displaystyle\int}

\newcommand{\wcM}{\mathrm{w}\hskip-0.6pt{\cM}}

\usepackage{comment}

\title[Boundedness of composition operators in Orlicz--Morrey spaces]
{
Boundedness of composition operators in Orlicz--Morrey spaces
}

\author{Masahiro Ikeda}
\address{
(Masahiro Ikeda) Graduate School of Information Science and Technology, The University of Osaka, 1-5, Yamadaoka, Suita-shi, Osaka 565-0871, Japan / Center for Advanced
Intelligence Project RIKEN, 1-4-1, Nihonbashi, Chuo-ku, Tokyo 103-0027, Japan.
}
\email{ikeda@ist.osaka-u.ac.jp/masahiro.ikeda@a.riken.jp}
\author{Isao Ishikawa}
\address{(Isao Ishikawa) 
 Center for Science Adventure and Collaborative Research Advancement (SACRA), Graduate School of Science, Kyoto University, Kitashirakawa Oiwake-cho, Sakyo-ku, Kyoto-shi, Kyoto 606-8502, Japan.
}
\email{ishikawa.isao.5s@kyoto-u.ac.jp}

\author{Ryota Kawasumi}
\address{
(Ryota Kawasumi)  Center for Mathematics and Data Science, Gunma University, 4-2 Aramaki-cho, Maebashi, 371-8510, Gunma, Japan.
}
\email{r-kawasumi@gunma-u.ac.jp}


\date{\today}


\begin{document}
\begin{abstract}
In this paper, we investigate the necessary and sufficient conditions for the boundedness of composition operators on Orlicz--Morrey spaces. Our results encompass Lebesgue and generalized Morrey spaces as special cases. Additionally, we characterize the boundedness of composition operators on weak Orlicz--Morrey spaces, which include the Orlicz--Morrey spaces. The key point of the proof is to identify the functions defining Orlicz--Morrey spaces, namely Young functions and growth functions (see Definition~\ref{defn:Young} and Definition~\ref{defn:cGdec}). 
Furthermore, our results also yield examples beyond the power-type Young function setting.
In particular, as discussed in Remark~\ref{rem:classical-morrey} (iii), by choosing appropriate growth and Young functions,
one obtains Orlicz--Morrey spaces whose defining Young function is not a power function. 
This observation also confirms that our approach genuinely generalizes the Morrey spaces established in~\cite{Hatano-Ikeda-Ishikawa-Sawano2021}.
\end{abstract}
\maketitle

\section{Introduction}

Let $n \in \N$.
Let $L^0(\R^n)$ be the space of all measurable functions on $\R^n$ and $|E|$ be the volume of a measurable set $E\subset \R^n$. Let $\psi:\R^n\to\R^n$ be a measurable map. 
We impose that $\psi$ is nonsingular, meaning $|\psi^{-1}(A)|=0$ for all $A \subset \R^n$ with $|A| = 0$.
The composition operator $C_\psi$ on a function space in $L^0(\R^n)$ is defined by
\[
C_\psi f \equiv f\circ \psi \quad\text{for all}\quad f \in L^0(\R^n).\] The composition operator is often called the Koopman operator and has recently been applied to data science and dynamical systems (see~\cite{Koopman1931, Koopman-Neumann1932, Bevanda-Sosnowski-Hirche2021, Crnjaric_Zic-Macesicy-Mezic2020, Hashimoto et. al2020,  Mezic2005, Ikeda-Ishikawa-Schlosser2022, Zhang-Zuazua2023}).

Characterizing boundedness of composition operators is a fundamental problem. Numerous studies have been conducted on various function spaces, including Lebesgue~\cite{Singh1976}, Orlicz~\cite{Cui-Hudzik-Kumar-Maligranda2004}, Morrey~\cite{Hatano-Ikeda-Ishikawa-Sawano2021}, Sobolev~\cite{Bourdaud-Sickel1999}, and Besov spaces~\cite{Ikeda-Ishikawa-Taniguchi2024}.

In this paper, we investigate the boundedness of composition operators on the Orlicz--Morrey spaces.
We extend previous work~\cite{Hatano-Ikeda-Ishikawa-Sawano2021} to the Orlicz--Morrey spaces in the sense of Sawano, Sugano, and Tanaka~\cite{Sawano-Sugano-Tanaka2011}. The Orlicz--Morrey spaces include Lebesgue and generalized Morrey spaces~(\cite{Guliyev2009, Mizuhara1991,Nakai1994, Sawano2019}) as special cases.

We denote by $B(a,r)$ the open ball centered at $a \in \R^n$ and with radius $r>0$, defined as
\[
B(a,r) \equiv \{x\in \R^n: |x-a| < r\}.
\]

First, we define the Young functions.
\begin{defn}[Young function]\label{defn:Young}
A function $\Phi:[0,\infty) \to [0,\infty)$ is called a Young function if $\Phi$ is convex, $\Phi(0)=0$ and $\dlim_{t\to \infty} \Phi(t) = \infty$.
Let $\iPy$ be the set of all Young functions.
\end{defn}

For $\Phi\in\iPy$,
we define the generalized inverse of $\Phi$
in the sense of O'Neil \cite[Definition~1.2]{ONeil1965}.

\begin{defn}[Generalized inverse of Young function]\label{defn:ginverse}
For $\Phi\in\iPy$ and $u\in[0,\infty)$, let
\begin{equation}\label{inverse}
\Phi^{-1}(u)
=
\inf\{t\ge0: \Phi(t)>u\}.
\end{equation}
\end{defn}
\noindent
Let $\Phi\in\iPy$. If $\Phi$ is bijective from $(0,\infty)$ to itself,
then $\Phi^{-1}$ is the usual inverse function of $\Phi$.
Moreover, we see
\begin{equation}\label{inverse ineq}
\Phi(\Phi^{-1}(u)) \le u \le \Phi^{-1}(\Phi(u))
\quad\text{for all $u\in[0,\infty)$},
\end{equation}
which is a generalization of Property 1.3 in \cite{ONeil1965}.

Next, we consider classes of growth functions $\vp:(0,\infty)\to(0,\infty)$.
\begin{defn}[Doubling condition]
Let $\vp: (0,\infty) \to (0,\infty)$.  The doubling condition on $\vp$ with a doubling constant $C_3\ge 1$ is that $\vp$ satisfies
\begin{equation}\label{doubling cond}
\frac1{C_3} \le \frac{\vp(r)}{\vp(s)} \le C_3 \quad \text{if} \quad \frac12 \le \frac{r}s \le 2.
\end{equation}
\end{defn}

We define almost decreasing (resp. almost increasing) functions below.
\begin{defn}[Almost decreasing and almost increasing function]
A function $f:(0,\infty) \to (0,\infty)$ is called an almost decreasing {\rm (}resp. almost increasing{\rm )} function, if there exist constants $C_1, C_2>0$ such that
for all $r,s\in(0,\infty)$,
\begin{equation}\label{eq:almost_decreasing}
C_1f(r)\ge f(s),\quad ({\rm resp.}\ f(r)\le C_2f(s)), \quad\text{if} \quad r<s.
\end{equation}
\end{defn}
\begin{defn}\label{defn:cGdec}
\begin{enumerate}
\item
Let $\cG_0$ be the set of all functions $\vp$ satisfying
\[
\lim_{r\to 0} \vp(r) = \infty \quad\text{and}\quad \lim_{r\to \infty} \vp(r) = 0.
\]
\item
A function $\vp$ is said to be a submultiplicative function if there exists a constant $C_4> 0$ such that
for all $r,s\in(0,\infty)$,
\begin{equation}\label{eq:submultiplicative}
\vp(rs) \le C_4\vp(r)\vp(s).
\end{equation}
Let $\cGdec_1$ be the set of all almost decreasing and submultiplicative functions.
\item
Let $\cGdec_2$ be the set of all functions $\vp:(0,\infty)\to(0,\infty)$
such that $\vp \in \cG_0\cap\cGdec_1$ and there exists a constant $C_5> 0$ such that
\[
\vp\left(\frac1r\right) \le C_5 \frac1{\vp(r)}\quad\text{for} \quad r>0.
\]
\end{enumerate}
\end{defn}
If $\vp\in\cGdec_1$ then $\vp$ satisfies the doubling condition. In fact, for all $r \in(0,\infty)$,
\begin{equation}\label{doubling cond2}
C_1^{-1}\vp(2r) \le \vp(r) \le C_4\vp\left(\frac12\right)\vp(2r),
\end{equation}
where we use the constant $C_1>0$ in \eqref{eq:almost_decreasing} and $C_4>0$ in \eqref{eq:submultiplicative}. 
Therefore, by choosing $\max(C_1, C_4\vp(1/2), 1)\le C_3$ and setting $r =s/2$ in \eqref{doubling cond2}, we obtain \eqref{doubling cond}.

We define the Orlicz--Morrey spaces on $\R^n$ below.
\begin{defn}[{\rm Orlicz--Morrey space, see~\cite[Definition~2.3]{Sawano-Sugano-Tanaka2011}}]\label{defn:OrliczMorrey}
For $a \in \R^n$, $r>0$, and $\Phi \in \iPy$, let
\begin{align*}
\|f\|_{\Phi,B(a,r)}
&\equiv
\inf\left\{ \lambda>0:
\frac{1}{|B(a,r)|}
\int_{B(a,r)} \!\Phi\!\left(\frac{|f(x)|}{\lambda}\right) dx \le 1
\right\}.
\end{align*}
Given $\vp$ satisfying the doubling condition, 
let $\cM_\Phi^\vp(\R^n)$ be the set of all functions $f$ such that the
following functional is finite:
\begin{align*}
\|f\|_{\cM_\Phi^\vp}
&\equiv
\sup_{a \in \R^n,\, r>0} \frac1{\vp(r)} \|f\|_{\Phi,B(a,r)}.
\end{align*}
\end{defn}
The Orlicz--Morrey spaces are quasi-normed spaces and $\cM_\Phi^\vp(\R^n) \neq \{0\}$. Here, quasi-normed spaces satisfy positivity, homogeneity, and the quasi-triangle inequality. 

Let $a, b \in \R^n$ and $r, s >0$. 
Note that the norm inequality \[ \frac1{|B(a,r)|}\int_{B(a,r)} \Phi\left(\frac{f}{C\|f\|_{\Phi, B(b,s)}}\right) dx \le 1 \] holds for some $C>0$ if and only if $ \|f\|_{\Phi, B(a,r)} \le C\|f\|_{\Phi, B(b,s)} $ holds for some $C>0$.

By specifying particular functions $\vp$ and $\Phi$, Orlicz--Morrey spaces encompass various function spaces as follows:
\begin{rem}
Put $v_n \equiv |B(0,1)|$ and $1\le q \le p < \infty$.
\begin{enumerate}
\item
If $\vp(r) = v_n^{-1/p}r^{-n/p} = |B(a,r)|^{-1/p}$, then we denote
$\cM_\Phi^\vp(\R^n)$
by
$\cM_\Phi^p(\R^n)$. The space $\cM_\Phi^p(\R^n)$ is the set of all functions $f\in L^0(\R^n)$ such that the following functional is finite:
\begin{align*}
\left\|f\right\|_{\cM_\Phi^p} &\equiv \sup_{a \in \R^n,\, r>0} |B(a,r)|^{1/p} \|f\|_{\Phi,B(a,r)}.
\end{align*}
\item
If $\Phi(t)=t^q$,
then we denote
$\cM_\Phi^\vp(\R^n)$ by $\cM_q^\vp(\R^n)$,
which are called generalized Morrey spaces. $\cM_q^\vp(\R^n)$ is  the set of all functions $f\in L^0(\R^n)$ such that the following functional is finite:
\begin{align*}
\left\|f\right\|_{\cM_q^\vp} &\equiv \sup_{a \in \R^n,\, r>0} \frac1{\vp(r)}\left(\frac1{|B(a,r)|}\int_{B(a,r)} |f(y)|^qdy \right)^{1/q}.
\end{align*}
\item
If $\vp(r)=v_n^{-1/p}r^{-n/p}$ and $\Phi(t) = t^q$, then $\cM_\Phi^\vp(\R^n)=\cM_q^p(\R^n)$. The space 
$\cM_q^p(\R^n)$ is called a Morrey space and is defined by
\[
\cM_q^p(\R^n) \equiv \{ f \in L^0(\R^n) : \|f\|_{\cM_q^p} < \infty\}
\]
endowed with the norm
\[
\|f\|_{\cM_q^p} \equiv \sup_{a\in \R^n, r>0} |B(a,r)|^{1/p-1/q} \left(\int_{B(a,r)} |f(y)|^q dy\right)^{1/q}.
\]
\end{enumerate}
\end{rem}


We now state the main results of the present paper. The following theorem provides a sufficient condition for the boundedness of the composition operator on Orlicz--Morrey spaces. Before that, we define a Lipschitz map.

A map $\psi:\R^n\to\R^n$ is called a Lipschitz map if there exists $L>0$ such that
\[
|\psi(x) -\psi(y)| \le L|x-y| \quad \text{for} \quad x, y \in \R^n.
\]

\begin{thm}\label{thm:OM-OM-second-sufficiency}
Let $\Phi\in \iPy$ and $\vp$ satisfy the doubling condition.
Assume that $\psi:\R^n\to \R^n$ is a measurable nonsingular transformation. If $\psi$ is a Lipschitz map that satisfies the volume estimate
\begin{equation}\label{Operator-norm-weak}
|\psi^{-1}(A)| \le K|A|
\end{equation}
for all measurable sets $A \subset \R^n$ with some constant $K>0$, then $C_\psi$ is bounded in $\cM_\Phi^\vp(\R^n)$, more precisely, there exists $C>0$ such that
\begin{equation*}
\left\| C_\psi f \right\|_{\cM_\Phi^\vp} \le K_0C \left\| f \right\|_{\cM_\Phi^\vp},
\end{equation*}
where $K_0 = \max(1, K)$ and $C>0$ is independent of K.
\end{thm}

Conversely, the following theorem states that $\psi$ and $\psi^{-1}$ are Lipschitz under the assumption of $\cM_\Phi^\vp(\R^n)$-boundedness of $C_\psi$ and $C_{\psi^{-1}}$.
Here, let ${\mathcal H}_\Phi^{p,\vp}$ be the set of all triples $(\Phi, \vp, p)$ where  $\Phi\in \iPy$, $\vp:(0,\infty)\to(0,\infty)$, and $p\ge 1$ such that there exist $C_5, C_6>0$ satisfying $\Phi(t) \le C_5t^p$ for $t\ge 1$ and $r^{-n/p} \le C_6\vp(r)$ for $r>0$.

For a measurable set $A \subset \R^n$, we denote by $\chi_A$ the characteristic function of $A$.
\begin{thm}\label{thm:OM-OM-second-necessary}
Let $\psi:\R^n\to \R^n$ be a diffeomorphism in the sense that $\psi$ and its
inverse $\psi^{-1}$ are differentiable.
Let $\Phi\in \iPy$, $\vp \in \cGdec_2$ and $p\ge 1$ be such that $ (\Phi, \vp, p)\in {\mathcal H}_\Phi^{p,\vp}$ and $\chi_{[0,1]\times \R^{n-1}} \in \cM_\Phi^\vp(\R^n)$.
 Suppose that $\Phi$ and $\vp$ satisfy either one of the following conditions.
\begin{enumerate}
\item $n=1$.
\item Let $n \ge 2$ and $k \in [1,n]\cap\N$. Suppose that the families
\begin{align*}
\mathcal F_{k-1}
&\equiv
\left\{
r\mapsto \frac1{\vp(r)\Phi^{-1}(C_7r^{k-1})}: C_7>0
\right\},\\
\mathcal F_k
&\equiv
\left\{
r\mapsto \frac1{\vp(r)\Phi^{-1}(C_7r^k)}: C_7>0
\right\}
\end{align*}
are almost increasing and almost decreasing, respectively,
with constants independent of $C_7$.
\end{enumerate}

If $C_\psi$ and $C_{\psi^{-1}}$ induced by $\psi$ and $\psi^{-1}$, respectively, are bounded on $\cM_\Phi^\vp(\R^n)$, then $\psi$ and $\psi^{-1}$ are Lipschitz.
\end{thm}

As a note to Theorem~\ref{thm:OM-OM-second-necessary}, we mention the following.
\begin{rem}\label{rem:classical-morrey}
\begin{enumerate}
\item The assumption $\chi_{[0,1]\times\R^{n-1}}\in \cM_\Phi^\vp(\R^n)$ is included in Theorem~\ref{thm:OM-OM-second-necessary} because, for $n\ge 2$, we have $\chi_{[0,1]\times\R^{n-1}}\not\in L^p(\R^n)$. Thus Theorem~\ref{thm:OM-OM-second-necessary} applies to functions beyond $L^p(\R^n)$ {\rm (}see also \cite[Remark~1.7]{Hatano-Ikeda-Ishikawa-Sawano2021}{\rm )}. For a proof that $\chi_{[0,1]\times \R^{n-1}}\in \cM_\Phi^\vp(\R^n)$, see Lemma~\ref{lem:chi-est} below.
\item
The important point in Theorem~\ref{thm:OM-OM-second-necessary} is that the theorem covers  \cite[Theorem~1.6]{Hatano-Ikeda-Ishikawa-Sawano2021} satisfying $1 \le q < p < \infty$. 
This is because, if $\vp(r) = r^{-n/p}$ and $\Phi(t) = t^q$ with $1 \le q < p < \infty$, then the assumptions of Theorem~\ref{thm:OM-OM-second-necessary} holds and $\cM_\Phi^\vp(\R^n) = \cM_q^p(\R^n)$. Further,  (ii) in Theorem~\ref{thm:OM-OM-second-necessary} implies that
\[
\frac{n}{k} q< p < \frac{n}{k-1}q
\]
if $k\ge 2$ and $nq < p$ if $k=1$. Therefore, the inequalities as above are used in~\cite{Hatano-Ikeda-Ishikawa-Sawano2021}.
\item
An example of a pair $(\vp,\Phi)$, other than the power-type Young-function case,
satisfying the assumptions of Theorem~\ref{thm:OM-OM-second-necessary} is the following.
This example is outside the pure power Young function setting.
Let $e \le q \le p$. Set $\vp(r) = r^{-n/p}$ and
\[
\Phi(t)=
\begin{cases}
0, & t=0,\\
\exp(-2/t), & 0< t< 1,\\
e^{-2}t^q, & t \ge 1.
\end{cases}
\]
Moreover,  putting $\rho = 2p/q$, 
\begin{equation}\label{eq:appendix_A}
\cM_2^\rho(\R^n)\cap L^\infty(\R^n) \subset\cM_\Phi^\vp(\R^n),
\end{equation}
where $L^\infty(\R^n)$ is the set of all measurable essentially bounded functions $f$ with the norm
\[
\left\|f\right\|_{L^\infty} \equiv \inf\left\{t>0: |\{x\in \R^n:|f(x)| >t\}| =0\right\}.
\]
The above inclusion follows from the estimate provided in \eqref{eq:appendix_A} in the Appendix.
\end{enumerate}
\end{rem}

If $\vp \notin \cG_0$ (hence $\vp \notin \cGdec_2$), then the following assertion
holds.
\begin{rem}{\rm \cite[Example~44]{Sawano-Fazio-Hakim2020}}\label{rem:Linf}
Let $f \in \cM_\Phi^\vp(\R^n)$. If $0 < \inf_{r>0} \vp(r) \le \sup_{r>0} \vp(r) < \infty$, then $C_\psi$ is bounded in $\cM_\Phi^\vp(\R^n)$.
This claim follows from the estimate provided in \eqref{eq:appendix_B} in the Appendix. 
\end{rem}

For a measurable set $A \subset \R^n$, $f \in L^0(\R^n)$, and $t>0$, we define
\[
m(A, f, t) \equiv|\{x \in \R^n: |f(x)|\chi_A(x) >t\}| = |\{x \in A: |f(x)| >t\}|.
\]
In the case $A = \R^n$, we briefly denote it by $m(f,t)$.

The following theorem states the boundedness of the composition operator on the weak Orlicz--Morrey spaces~(see \cite{Hakim-Sawano2016, Kawasumi-Nakai-Shi2023}, for example).
Before that,  we define the weak Orlicz--Morrey spaces using $m(A, f, t)$. 
\begin{defn}[Weak Orlicz--Morrey space]\label{defn:weakOrliczspace}
For $a \in \R^n$, $r>0$, and $\Phi \in \iPy$, let
\begin{align*}
\|f\|_{\Phi,B(a,r),\weak}
&\equiv
\inf\left\{ \lambda>0:
\sup_{t\in(0,\infty)} \frac{1}{|B(a,r)|} \Phi\!\left(\frac{t}\lambda\right)m(B(a,r), f, t)\le 1
\right\}.
\end{align*}
Given $\vp$ satisfying the doubling condition, let $\wcM_\Phi^\vp(\R^n)$ be the set of all functions $f \in L^0(\R^n)$ such that the
following functional is finite:
\begin{align*}
\|f\|_{\wcM_\Phi^\vp}
&\equiv
\sup_{a \in \R^n,\, r>0} \frac1{\vp(r)} \|f\|_{\Phi,B(a,r), \weak}.
\end{align*}
\end{defn}
The weak Orlicz--Morrey spaces contain the Orlicz--Morrey spaces. This claim means $\|f\|_{\wcM_\Phi^\vp} \le \|f\|_{\cM_\Phi^\vp}$. Hence, for many problems, it is expected that the weak Orlicz--Morrey spaces will be more useful than the Orlicz--Morrey spaces in dynamical systems.
\begin{thm}\label{thm:wOM-wOM}
Let $\Phi\in \iPy$, $\vp \in \cGdec_1$, and let $\psi:\R^n\to\R^n$ be a measurable function. Then, the composition operator $C_\psi$ induced by $\psi$ is bounded on $\wcM_\Phi^\vp(\R^n)$
if and only if there exists a constant $K>0$ such that for all measurable sets $A$ in $\R^n$, the following estimate holds:
\begin{equation}\label{eq:chi-ineq}
\left\|\chi_{\psi^{-1}(A)}\right\|_{\cM_\Phi^\vp} \le K \left\|\chi_{A}\right\|_{\cM_\Phi^\vp}.
\end{equation}
\end{thm}
\begin{rem}
If $\wcM_\Phi^\vp(\R^n)=\cM_\Phi^\vp(\R^n)=L^\infty(\R^n)$,
then condition~\eqref{eq:chi-ineq} reduces to the nonsingularity of $\psi$.
Under this condition, $C_\psi$ is bounded on $L^\infty(\R^n)$.
\end{rem}

The organization of this paper is as follows: 
In Section~2, we give the proof of Theorem~\ref{thm:OM-OM-second-sufficiency}.
In Section~3, we prepare the proof of  Theorem~\ref{thm:OM-OM-second-necessary}. In Section~4, we give the proof of Theorem~\ref{thm:OM-OM-second-necessary}. In Section~5, we give the proof of Theorem~\ref{thm:wOM-wOM}.

At the end of this section, we make some conventions. Throughout this paper, we always use $C$ to denote a positive constant.
If $f \le Cg$, we then write $f\ls g$ or $g\gs f$, and if $f \ls g \ls f$, we then write $f \sim g$.

\section{Proof of Theorem~\ref{thm:OM-OM-second-sufficiency}}

Before proving Theorem~\ref{thm:OM-OM-second-sufficiency}, we state the following assertion:
\begin{prop}\rm\cite[Theorem~2.1]{Cui-Hudzik-Kumar-Maligranda2004}\label{Bounded-Modular}
Assume that $\psi: \R^n\to \R^n$ is a measurable nonsingular transformation.
There exists some $K>0$  such that
\begin{equation}\label{eq:modular}
\int_{\R^n} \Phi\left( f(\psi(x)) \right) dx \le K \int_{\R^n} \Phi\left( f(x) \right) dx
\end{equation}
holds for all $f \in L^0(\R^n)$ satisfying $\dint_{\R^n} \Phi\left( f(x) \right) dx < \infty$ 
if and only if
\begin{equation*}
|\psi^{-1}(A)| \le K|A|
\end{equation*}
for all $A \subset \R^n$. 
\end{prop}

\begin{proof}[ Proof of Theorem~\ref{thm:OM-OM-second-sufficiency}]
Set a ball $B_0 = B(a,r)$ with $a \in \R^n$ and $r>0$.
Let $x, y \in B_0$ be fixed. 
We note that the Lipschitz continuity of $\psi$ implies the validity of the following inequalities.
\begin{align}\label{eq:Lipschitz}
\notag{\rm diam}(\psi(B_0)) &\equiv \sup_{x, y \in B_0}|\psi(x) -\psi(y)| \le L\sup_{x, y \in B_0}|x-y|
\\&\le L\sup_{x, y \in B_0}(|x-a| + |y-a|) \le 2Lr.
\end{align}
Thus, there exists $B_1= B(\psi(a), Lr)$ such that
\[
B_1 \supset \psi(B_0) \quad {\rm and \quad } |B_1| = L^n |B_0|,
\]
because of  the Lipschitz continuity of $\psi$.
Set $M =\max(1,L^n)$ and  let $K\ge 1$ satisfy \eqref{eq:modular}.
By exploiting the convexity of $\Phi$, we can prove the following inequalities
\begin{align*}
\frac1{|B_0| }\int_{B_0} \!\Phi\!\left(\frac{|C_\psi f(x) | }{KM\left\|f\right\|_{\Phi, B_1} }\right)\! dx
&\le
\frac1{|B_0| }\int_{\R^n} \!\Phi\!\left(\frac{|f(\psi(x)) |}{KM\left\|f\right\|_{\Phi, B_1} }\right) \chi_{\psi(B_0)}(\psi(x)) \! dx
\\&=
\frac1{|B_0| }\int_{\R^n} \!\Phi\!\left(\frac{|f(\psi(x))| \chi_{\psi(B_0)}(\psi(x)) } {KM\left\|f\right\|_{\Phi, B_1} }\right) \! dx
\\ &\le
\frac{L^n}{|B_1| }\int_{B_1} \!\Phi\!\left(\frac{|f(x) |}{M\left\|f\right\|_{\Phi, B_1} }\right)\! dx
\le 1,
\end{align*}
where we use Proposition~\ref{Bounded-Modular} for the third inequality. As a result, we have $ \left\| C_\psi f \right\|_{\Phi, B_0} \le KM\left\| f \right\|_{\Phi, B_1}$. Therefore, by the doubling condition of $\vp$, we have
\begin{align*}
\frac1{\vp(r)} \left\|C_\psi f \right\|_{\Phi, B_0}
&\le
\frac{\vp(Lr)}{\vp(r)}\frac1{\vp(Lr)} KM\left\| f \right\|_{\Phi, B_1}
\\&
\le
K\max(1,L^n)\left(\sup_{r>0}\frac{\vp(Lr)}{\vp(r)}\right)\left\| f \right\|_{\cM_\Phi^\vp}
\\&\ls
KC\left\| f \right\|_{\cM_\Phi^\vp},
\end{align*}
where $C$ is a positive constant depending on $L$ and independent of  $K$.
Therefore, we have the desired conclusion.
\end{proof}

\section{Preliminaries of Theorem~\ref{thm:OM-OM-second-necessary}}

In this section, we prepare some lemmas for the proof of Theorem~\ref{thm:OM-OM-second-necessary}. First, we define the notions.  
\begin{itemize}
\item
Let $M_n(\R)$ be the set of all $n\times n$ regular matrices.
\item
The space $\Lic(\R^n)$ stands for the set of all $\Li(\R^n)$ functions with compact support.
\item
The space $\Cic(\R^n)$ is the set of all smooth functions with compact support.
\item
For a differentiable vector-valued function $\psi = (\psi_1, \dots, \psi_n)^T$ on $\R^n$, we denote by $D\psi$
the Jacobian matrix of $\psi$, that is,
\[
D\psi \equiv \left( \frac{\partial\psi_i}{\partial x_j} \right)_{1\le i,j\le n} \equiv (\psi_{i,j})_{1\le i,j\le n}.
\]
\item 
Let $A \in M_n(\R)$. 
Denote $\left\|A\right\|_{\rm Fro}$ by the Frobenius norm of $A$, that is, 
\[
\left\|A\right\|_{\rm Fro} \equiv \left(\sum_{i,j=1}^n a_{ij}^2\right)^{1/2},
\]
where $a_{ij}$ is the $(i,j)$-element in $A$.  
\item
Let $E$ be a normed space.
For a linear operator $T$ from $E$ to itself, we define the operator norm by
\begin{equation}\label{eq:operator_norm}
\left\|T\right\|_{E \to E} \equiv \sup_{f\neq 0}\frac{ \|T f\|_{E} }{ \|f\|_{E}} \equiv \sup_{\|f\|_{E} = 1}\|Tf\|_{E}.
\end{equation}
\end{itemize}

If $\Phi_2(t) = t^p$ and $\vp_2(r) = r^{-n/p}$ with $p\ge 1$ in \cite[Remark~4]{Sawano-Fazio-Hakim2020}, then we have the following assertion. 
\begin{lem}
\label{lem:embedding_in_Lebesgue}
Let $\Phi \in \iPy$, $\vp \in \cGdec_1$, and $p \ge 1$. If $ (\Phi, \vp, p)\in {\mathcal H}_\Phi^{p,\vp}$, then
\[
\|f\|_{\cM_{\Phi}^{\vp}}
\le C \|f\|_{L^p}
\]
for all $f \in L^p(\R^n)$.
\end{lem}

Using an operator norm, we obtain the following lemmas, which are used in the proof of Lemma~\ref{lem:diag}.
\begin{lem}\label{lem:operator-norm}
Let $\psi(x) = cx$ with  $c >0$, $\Phi \in \iPy$ and $\vp\in \cG_0\cap\cGdec_1$. Then
\begin{equation}\label{eq:operator-norm}
\inf_{r>0} \frac{\vp\left(cr\right)}{\vp(r)} \le
\left\|C_\psi\right\|_{\cM_\Phi^\vp\to \cM_\Phi^\vp}
\le
\sup_{r>0} \frac{\vp\left(cr\right)}{\vp(r)}.
\end{equation}
Moreover, if $\vp \in \cGdec_2$, then
\begin{equation}\label{eq:operator-norm2}
\left\|C_\psi\right\|_{\cM_\Phi^\vp\to \cM_\Phi^\vp} \sim \vp(c).
\end{equation}
\end{lem}

\begin{proof}
We show only the left inequality in \eqref{eq:operator-norm} because the right inequality is proved similarly, and \eqref{eq:operator-norm2} is obtained immediately from \eqref{eq:operator-norm}.
Let $\Phi \in \iPy$, $a \in \R^n$, $c \in \R\setminus\{0\}$, and $r>0$. Put $B_{(c)} \equiv B(ca, cr)$. Then, we compute
\begin{align*}
\frac1{|B_{(c)} |} \int_{B_{(c)} } \Phi\left(\frac{|f(x)|}{\left\|C_\psi f\right\|_{\Phi, B_{(1)} }}\right) dx
&=
\frac1{|B_{(c)} |} \int_{B_{(1)} } \Phi\left(\frac{|f(cy)|}{\left\|C_\psi f\right\|_{\Phi, B_{(1)} }}\right) c^ndy
\\&=\frac1{|B_{(1)} |} \int_{B_{(1)} } \Phi\left(\frac{|f(cy)|}{\left\|C_\psi f\right\|_{\Phi, B_{(1)} }}\right) dy \le 1.
\end{align*}
As a result, we have
\begin{align*}
\frac1{\vp(cr)}\left\| f\right\|_{\Phi, B_{(c)} } &\le
\frac{\vp(r)}{\vp(cr)}\frac1{\vp(r)}\left\| C_\psi f\right\|_{\Phi, B_{(1)} } \le \sup_{r>0}\frac{\vp(r)}{\vp(cr)} \left\| C_\psi f\right\|_{\cM_\Phi^\vp}
\end{align*}
and
obtain
\[
\inf_{r>0}\frac{\vp(cr)}{\vp(r)}
\left\| f\right\|_{\cM_\Phi^\vp} \le \left\| C_\psi f\right\|_{\cM_\Phi^\vp}.
\]

Furthermore, by the submultiplicativity of the function $\vp$, we have
\begin{equation}\label{eq:left(1.3)}
\frac{\vp(cr)}{\vp(r)} \le C_4\vp(c).
\end{equation}
On the other hand, since $\vp \in \cGdec_2$, we have
\begin{equation}\label{eq:right(1.3)}
\frac{\vp(cr)}{\vp(r)} \ge C_5^{-1} \vp(cr)\vp\left(\frac1r\right) \ge C_4^{-1}C_5^{-1} \vp(c).
\end{equation}
Combining \eqref{eq:left(1.3)} and \eqref{eq:right(1.3)} with  \eqref{eq:operator-norm}, we obtain \eqref{eq:operator-norm2}.

Therefore, we have the desired conclusion.
\end{proof}


We consider the operator norm of the composition operator  defined by an orthogonal matrix.
\begin{lem}\label{lem:Orth-op}
Let $\Phi \in \iPy$ and $\vp \in \cGdec_1$. If $W \in M_n(\R)$ is an orthogonal matrix, then
\[
\left\|C_W\right\|_{\cM_\Phi^\vp\to \cM_\Phi^\vp} = 1.
\]
\end{lem}
\begin{proof}
We show only
$
\left\|C_W f\right\|_{\cM_\Phi^\vp} \le\left\| f\right\|_{\cM_\Phi^\vp},
$
 because
$\left\|C_W f\right\|_{\cM_\Phi^\vp} \ge\left\| f\right\|_{\cM_\Phi^\vp}$
is similarly proved.
Let $a \in \R^n$ and $r>0$. Put $B_0 = B(a,r)$ and $B_W = B(Wa, r)$. We estimate
\begin{align*}
\frac1{|B_0|} \int_{B_0} \Phi\left(\frac{|C_W f(x)|}{\left\|f\right\|_{\Phi, B_W}}\right) dx
&=
\frac1{|B_0|} \int_{B_W} \Phi\left(\frac{|f(y)|}{\left\|f\right\|_{\Phi, B_W}}\right) |{\rm det}W^{-1}|dy
\\&= \frac1{|B_W|} \int_{B_W} \Phi\left(\frac{|f(y)|}{\left\|f\right\|_{\Phi, B_W}}\right) dy
\le 1,
\end{align*}
where we use that  if  $x \in B_0$ then $Wx \in B_W$ for the first equality. As a result, we have
$
\left\| C_Wf\right\|_{\Phi, B_0} \le \left\| f\right\|_{\Phi, B_W}. 
$
Therefore, we have the desired conclusion.
\end{proof}



We estimate the norm inequality for a function scaled by the Jacobian matrix.
\begin{lem}\label{lem:Jacobi-bdd2}
Let $\Phi \in \iPy$, $\vp \in \cGdec_2$, and $p \ge 1$. Suppose that a diffeomorphism $\psi:\R^n\to \R^n$ induces a bounded operator $C_\psi$ from $\cM_\Phi^\vp(\R^n)$ to itself. If $ (\Phi, \vp, p)\in {\mathcal H}_\Phi^{p,\vp}$, then the following inequality holds{\rm :}
\[
\left\|f(D\psi(x_0)\cdot)\right\|_{\cM_\Phi^\vp} \ls \left\|f\right\|_{\cM_\Phi^\vp}
\]
for all $x_0 \in \R^n$ and $f \in \cM_\Phi^\vp(\R^n)$.
\end{lem}

\begin{proof}
We divide the proof into the following three cases.

{\bf Case~(i) $(f \in \Cic(\R^n))$}
Let $a \in \R^n$ and $r, t>0$. Denote
\[
B_0 = B(a,r), \quad B_1 = B(x_0+at, tr), \quad \text{and} \quad B_2 = B\left(a +\frac{x_0 -\psi(x_0)}t,r\right).
\]
Notice that
\[
t^n|B_0 |= |B_1| = t^n|B_2| = v_nt^nr^n.
\]
Let $K=\|C_\psi\|_{\cM_\Phi^\vp\to\cM_\Phi^\vp}$.
Putting
\[
g(\cdot) = f\left(\frac{\psi(\cdot) -\psi(x_0)}t\right),
\]
we estimate
\begin{align*}
&
\frac1{ |B_0| }\int_{B_0} \!\Phi\left(
\frac{1}{\left\|g\right\|_{\Phi,B_1 }}
\left|
g\!\left(x_0+ tx\right)
\right|\right)\! dx
=
\frac{t^n}{ |B_1| }\int_{B_1} \!\Phi\left(\frac{1}{\left\|g\right\|_{\Phi, B_1}}\left|g\!\left(y\right)\right|\right)\! t^{-n}dy \le 1,
\end{align*}
where $y= x_0 + tx$. Thus,
\begin{align}
\frac1{\vp(r)}
\left\|
g\!\left(x_0+ t\cdot\right)
\right\|_{\Phi, B_0}
\ls
\frac{\vp(t)}{\vp(tr)}
\left\|
g\!\left(\cdot\right)
\right\|_{\Phi, B_1}
\notag&\le
\vp(t)
\left\|g\!\left(\cdot\right)\right\|_{\cM_\Phi^\vp}
\notag\\ &\le
K\vp(t)
\left\|f\left(
\frac{\cdot -\psi(x_0)}t
\right)\right\|_{\cM_\Phi^\vp}, \label{eq:bdd-1/2}
\end{align}
where we use the submultiplicativity of $\vp$ for the first inequality
and the boundedness of $C_\psi$ on $\cM_\Phi^\vp(\R^n)$ for the third inequality. \eqref{eq:bdd-1/2} implies that
\begin{equation}\label{eq:bdd-MM1/2}
\left\|f\left(
\frac{\psi(x_0+t\cdot) -\psi(x_0)}t
\right)\right\|_{\cM_\Phi^\vp}
\le
K\vp(t)
\left\|f\left(
\frac{\cdot -\psi(x_0)}t
\right)\right\|_{\cM_\Phi^\vp}.
\end{equation}
Meanwhile,
\begin{align*}
\frac{1}{ |B_1| }\int_{B_1} \!\Phi\left(\frac{1}{\left\|f\right\|_{\Phi, B_2}} \left|f\!\left(\frac{x -\psi(x_0)}{t}\right)\right|\right)\! dx
=
\frac{t^{-n}}{ |B_2| }\int_{B_2} \!\Phi\left(\frac{1}{\left\|f\right\|_{\Phi, B_2}} |f(z)|\right)\! t^ndz
\le 1,
\end{align*}
where $z = (x-\psi(x_0))/t$. As a result, we have
\begin{equation}\label{eq:norm_estimate}
\left\|f\left(\frac{\cdot -\psi(x_0)}{t}\right)\right\|_{\Phi, B_1}\le \left\|f\right\|_{\Phi, B_2}.
\end{equation}
Thus, we estimate
\begin{align*}
\frac1{\vp(tr)}\left\|f\!\left(\frac{\cdot -\psi(x_0)}{t}\right)\right\|_{\Phi, B_1}
&\le
C_4\frac{\vp(1/t)}{\vp(r)}
\left\|f\right\|_{\Phi, B_2}
\\&\le C_4C_5
\frac1{\vp(t)}
\frac{1}{\vp(r)}
\left\|f\right\|_{\Phi, B_2}
\\&\le C_4C_5
\frac1{\vp(t)}
\left\|f\right\|_{\cM_\Phi^\vp},
\end{align*}
where we use the submultiplicativity of the function $\vp$ and \eqref{eq:norm_estimate} for the first inequality, and $\vp(1/r) \ls 1/\vp(r)$ for the second inequality.
Hence,
\begin{equation}\label{eq:bdd-MM2/2}
\left\|f\!\left(\frac{\cdot -\psi(x_0)}{t}\right)\right\|_{\cM_\Phi^\vp}
\ls
\frac1{\vp(t)}
\left\|f\right\|_{\cM_\Phi^\vp}.
\end{equation}
Combining \eqref{eq:bdd-MM1/2} and \eqref{eq:bdd-MM2/2}, we have
\begin{align*}
\left\|f\left(
\frac{\psi(x_0+t\cdot) -\psi(x_0)}t
\right)\right\|_{\cM_\Phi^\vp}
\ls
K\left\|f\right\|_{\cM_\Phi^\vp}.
\end{align*}
Here, fix $\{t_n\}_{n=1}^\infty$ such that $t_n \to 0$ as $n\to \infty$. For $x \in B_0$,
\[
\left|
f\!\left(\frac{\psi(x_0+ t_nx) -\psi(x_0)}{t_n}\right)\right| \to \left|f(D\psi(x_0)x) \right|
\]
and, for $t>0$,
\[
m\left(f\!\left(\frac{\psi(x_0+ t_nx) -\psi(x_0)}{t_n}\right), t\right)\to m\left(f(D\psi(x_0)x), t\right)
\]
as $n \to \infty$.
Therefore,
\begin{align*}
&
\liminf_{n\to \infty} \int_{B_0} \Phi\left(\left|f\!\left(\frac{\psi(x_0+ t_nx) -\psi(x_0)}{t_n}\right)\right|\right) dx
\\&\ge
\int_{B_0}\liminf_{n\to \infty} \Phi\left(\left|f\!\left(\frac{\psi(x_0+ t_nx) -\psi(x_0)}{t_n}\right)\right|\right)dx
\\&\ge
\int_{B_0} \Phi\left(|f(D\psi(x_0)x)|\right) dx.
\end{align*}
By letting $t_n\to 0$, we obtain the desired result, that is
\[
\left\|f(D\psi(x_0)\cdot)\right\|_{\cM_\Phi^\vp} \ls \left\|f\right\|_{\cM_\Phi^\vp}.
\]
{\bf Case~(ii) $(f \in \Lic(\R^n))$}
For any $p \in [1, \infty)$, we can choose a sequence $\{f_j\}_{j=1}^\infty \subset \Cic(\R^n)$ such that $f_j$ converges to $f$ in $L^p(\R^n)$ as $j \to \infty$. By passing to
a subsequence, we may assume that $f_j$ converges to $f$, almost everywhere in $\R^n$ as $j\to \infty$.
Thus, by the Fatou lemma, the inequality
\[
\left\|f(D\psi(x_0)\cdot)\right\|_{\cM_\Phi^\vp} \le \liminf_{j\to \infty} \left\|f_j(D\psi(x_0)\cdot)\right\|_{\cM_\Phi^\vp}
\]
holds since
\[
\left\|f(D\psi(x_0)\cdot)\right\|_{\Phi, B_0} \le \liminf_{j\to \infty} \left\|f_j(D\psi(x_0)\cdot)\right\|_{\Phi, B_0}.
\]
As we have proved the assertion for $f_j$, we have
\[
\left\|f_j(D\psi(x_0)\cdot)\right\|_{\cM_\Phi^\vp} \ls 2KC\left\|f_j\right\|_{\cM_\Phi^\vp}.
\]
Because $L^p(\R^n)$ is embedded into $\cM_\Phi^\vp(\R^n)$, see Lemma~\ref{lem:embedding_in_Lebesgue}, $f_j$ converges to $f$ in $\cM_\Phi^\vp(\R^n)$ as $j \to \infty$. Consequently,
\[
\liminf_{j\to \infty} \left\|f_j\right\|_{\cM_\Phi^\vp} = \left\|f\right\|_{\cM_\Phi^\vp}.
\]
By combining these observations, the following estimate holds:
\[
\left\|f(D\psi(x_0)\cdot)\right\|_{\cM_\Phi^\vp} \ls K\left\|f\right\|_{\cM_\Phi^\vp}.
\]\\
{\bf Case~(iii) ($f \in \cM_\Phi^\vp(\R^n)$)}
For $k\in \N$, we define an element $f_k \in \Lic(\R^n)$ by
\[
f_k(x) = f(x) \chi_{[-k,k]^n}(x)\chi_{[0,k]}(|f(x)|) \quad \text{for all} \quad x \in \R^n.
\]
Then, we have
\[
\left\|f_k(D\psi(x_0)\cdot)\right\|_{\cM_\Phi^\vp} \ls K\left\|f_k\right\|_{\cM_\Phi^\vp} \le K\left\|f\right\|_{\cM_\Phi^\vp},
\]
by the previous paragraph. By using the Fatou lemma again, we obtain
\[
\left\|f(D\psi(x_0)\cdot)\right\|_{\cM_\Phi^\vp} \ls K\left\|f\right\|_{\cM_\Phi^\vp}
\]
as required.
\end{proof}

%
%


Let $\psi:\R^n\to\R^n$ be a diffeomorphism and $D\psi:\R^n\to M_n(\R)$ be its Jacobian matrix. For
$x_0 \in \R^n$, the Jacobian matrix $D\psi(x_0)$
can be decomposed by the singular value decomposition
as
\begin{equation*}
D\psi(x_0) = U\Sigma V,
\end{equation*} 
where 
\begin{equation}\label{eq:diag_a}
\Sigma= \Sigma(x_0) = {\rm diag}(\alpha_1(x_0), \dots, \alpha_n(x_0))
\end{equation} 
is a diagonal matrix with positive diagonal entries
satisfying $\alpha_1(x_0) \le\cdots \le\alpha_n(x_0)$, and $U=U(x_0)$ and $V = V(x_0)$ are orthogonal
matrices. 


Using Lemma 2.2 in \cite{Hatano-Ikeda-Ishikawa-Sawano2021}, 
we prove the following lemma.
\begin{lem}\label{lem:diag}
Let $a_1, \dots, a_n$ be a positive numbers i.e.,
\[
a_i > 0 \quad \text{for}\quad i=1, \dots, n.
\]
Set $D \equiv {\rm diag}(a_1,\dots,a_n)$. If $\Phi \in \iPy$ and $\vp \in \cGdec_2$, then the following estimate holds:
\[
\vp\left( \prod_{k=1}^n a_k\right)^{1/n}
\ls
\left\|C_D\right\|_{\cM_\Phi^\vp\to \cM_\Phi^\vp}.
\]
\end{lem}
\begin{proof}
We introduce the matrix $W \in M_n(\R)$ corresponding to the transformation
\[
(x_1, x_2,\dots, x_n) \to (x_2, x_3,\dots, x_n, x_1).
\]
Then, for any $k \in [1, \dots, n]$, we observe that
\[
W^{-k}DW^k = {\rm diag}(a_{n-k+1}, a_{n-k+2}, \dots, a_n, a_1, a_2, \dots, a_{n-k})
\]
holds, since $W^k$ maps the $l$-th elementary vector $e_l$ to $e_{l-k}$ if $l >k$ and $e_{l-k+n}$ otherwise.
Then
\[
\left\|C_{W^{-k}DW^k}\right\|_{\cM_\Phi^\vp\to\cM_\Phi^\vp}
\le
\left\|C_D\right\|_{\cM_\Phi^\vp\to\cM_\Phi^\vp}.
\]
On the other hand
\[
\left\|C_D\right\|_{\cM_\Phi^\vp\to\cM_\Phi^\vp}
\le
\left\|C_{W^{-k}DW^k}\right\|_{\cM_\Phi^\vp\to\cM_\Phi^\vp}
\]
holds. The identity
\[
\prod_{k=1}^nW^{-k}DW^k = \left(\prod_{k=1}^n a_k\right)E
\]
holds. By \eqref{eq:operator-norm2}, we have
\[
\left\|C_{\prod_{k=1}^nW^{-k}DW^k}\right\|_{\cM_\Phi^\vp\to \cM_\Phi^\vp} \sim \vp\left( \prod_{k=1}^n a_k\right).
\]
Combining this with the identity $C_{\prod_{k=1}^nW^{-k}DW^k} = \prod_{k=1}^n C_{W^{-k}DW^k}$,
\begin{align*}
\left\|C_{\prod_{k=1}^nW^{-k}DW^k}\right\|_{\cM_\Phi^\vp\to\cM_\Phi^\vp}
&=
\left\| \prod_{k=1}^n C_{W^{-k}DW^k}\right\|_{\cM_\Phi^\vp\to\cM_\Phi^\vp}
\\&\ls
\prod_{k=1}^n\left\|C_D\right\|_{\cM_\Phi^\vp\to\cM_\Phi^\vp}\le \left\|C_D\right\|_{\cM_\Phi^\vp\to\cM_\Phi^\vp}^n,
\end{align*}
where we use Lemma~\ref{lem:Orth-op} for the second inequality, 
which yields
\[
\left\|C_{W^k}\right\|_{\cM_\Phi^\vp\to\cM_\phi^\vp} = \left\|C_W\right\|_{\cM_\Phi^\vp\to\cM_\phi^\vp}^k =1 
\]
and
\[
\left\|C_{W^{-k}}\right\|_{\cM_\Phi^\vp\to\cM_\phi^\vp} = \left\|C_{W^{-1}}\right\|_{\cM_\Phi^\vp\to\cM_\phi^\vp}^k =1 
\]
for $k\in [1, n]\cap \N$.
Therefore,
\[
\vp\left( \prod_{k=1}^n a_k\right)^{1/n}
\ls
\left\|C_D\right\|_{\cM_\Phi^\vp\to \cM_\Phi^\vp}
\] holds. The conclusion of
this lemma is proved.
\end{proof}

\begin{lem}\label{lem:equiv-op}
Let $\Phi \in \iPy$ and $\vp \in \cGdec_1$. Let $A \in M_n(\R)$, and let
$\alpha_1,\dots,\alpha_n$ be the singular values of $A$.
Set $\Sigma = \mathrm{diag}(\alpha_1,\dots,\alpha_n)$.
Then
\[ 
\left\|C_\Sigma\right\|_{\cM_\Phi^\vp\to \cM_\Phi^\vp} = \left\|C_A\right\|_{\cM_\Phi^\vp\to \cM_\Phi^\vp}.
\]
\end{lem}

\begin{proof}
Let  $A \in M_n(\R)$ and write its singular value decomposition
$A = U \Sigma V$, where $U$ and $V$ are orthogonal matrices and
$\Sigma = \mathrm{diag}(\alpha_1,\dots,\alpha_n)$. By using Lemma~\ref{lem:Orth-op}, we estimate
\[
\left\|C_A\right\|_{\cM_\Phi^\vp\to \cM_\Phi^\vp}
=
\left\|C_{U\Sigma V}\right\|_{\cM_\Phi^\vp\to \cM_\Phi^\vp} \le \left\|C_{\Sigma}\right\|_{\cM_\Phi^\vp\to \cM_\Phi^\vp}
\]
and
\begin{align*}
\left\|C_{\Sigma}\right\|_{\cM_\Phi^\vp\to \cM_\Phi^\vp} 
&\le \left\|C_{U\Sigma V}\right\|_{\cM_\Phi^\vp\to \cM_\Phi^\vp} =\left\|C_A\right\|_{\cM_\Phi^\vp\to \cM_\Phi^\vp}.
\end{align*}
Therefore, we have the conclusion.
\end{proof}
Let $x_0 \in \mathbb{R}^n$.
Applying Lemma~\ref{lem:equiv-op} to $A = D\psi(x_0)$, and denoting by $\Sigma(x_0)$ the diagonal matrix of singular values of
$D\psi(x_0)$, we obtain
\begin{equation}\label{eq:equiv-op}
\| C_{\Sigma(x_0)} \|_{\cM_\Phi^\vp \to \cM_\Phi^\vp}
=
\| C_{D\psi(x_0)} \|_{\cM_\Phi^\vp \to \cM_\Phi^\vp}.
\end{equation}

\begin{prop}\label{prop:vp-const}
Let $\Phi \in \iPy$, $\vp \in \cGdec_2$, $p\ge 1$, $x_0\in \R^n$, and $\psi:\R^n \to \R^n$ be a diffeomorphism. Assume that $ (\Phi, \vp, p)\in {\mathcal H}_\Phi^{p,\vp}$.
If $C_\psi$ and $C_{\psi^{-1}}$ induced by $\psi$ and $\psi^{-1}$, respectively, are bounded on the Orlicz--Morrey space $\cM_\Phi^\vp(\R^n)$,
then we have
\begin{equation}\label{eq:condition-2}
\vp\left(\prod_{k=1}^n \alpha_k(x_0) \right) \sim1,
\end{equation}
where $\alpha_1(x_0), \dots, \alpha_n(x_0)$ are the singular values of $D\psi(x_0)$.
\end{prop}
\begin{proof}
Let $x_0 \in \R^n$ and $U, U', V, V'$ be orthogonal matrices. Put $D\psi(x_0) = U \Sigma(x_0) V$.
By Lemma~\ref{lem:Jacobi-bdd2}, Lemma~\ref{lem:diag} and \eqref{eq:equiv-op}, we have
\begin{align*}
\vp\left(\prod_{k=1}^n \alpha_k(x_0) \right)^{1/n}
&\le
\left\|C_{\Sigma(x_0)}\right\|_{\cM_\Phi^\vp\to \cM_\Phi^\vp}
\sim
\left\|C_{D\psi(x_0)}\right\|_{\cM_\Phi^\vp\to \cM_\Phi^\vp}\le KC.
\end{align*}

Meanwhile, put $D\psi^{-1}(x_0) = U' \Sigma^{-1}(\psi^{-1}(x_0)) V'$. By Lemma~\ref{lem:Jacobi-bdd2}, Lemma~\ref{lem:diag} and \eqref{eq:equiv-op} applied to $\psi^{-1}$, we have
\begin{align*}
\vp\left(\prod_{k=1}^n \alpha_k(\psi^{-1}(x_0))^{-1} \right)^{1/n}
&\le
\left\|C_{\Sigma^{-1}(\psi^{-1}(x_0)) }\right\|_{\cM_\Phi^\vp\to \cM_\Phi^\vp}
\\&\sim
\left\|C_{D\psi^{-1}(x_0)}\right\|_{\cM_\Phi^\vp\to \cM_\Phi^\vp}\le KC.
\end{align*}
Note that singular values of $D\psi^{-1}(x_0)$ are
\[\alpha_1(\psi^{-1}(x_0))^{-1}, \dots, \alpha_n(\psi^{-1}(x_0))^{-1}\]
since $D\psi^{-1}(x_0) = [D\psi(\psi^{-1}(x_0))]^{-1} $.
As a result,
\[
\frac1{KC} \ls \vp\left(\prod_{k=1}^n \alpha_k(\psi^{-1}(x_0)) \right)^{1/n}
\]
holds. Consequently, we have \eqref{eq:condition-2}.
\end{proof}
Further, by the assumption of $\vp \in \cGdec_1$, a condition on $\vp$ that implies \eqref{eq:condition-2} can be stated clearly as follows:
\begin{rem}\label{rem:vp-const}
Let $\vp \in \cGdec_1$.
If $\vp(\alpha_1(x_0)) \ls 1$, then the upper estimate 
\[
\vp\left(\prod_{k=1}^n \alpha_k(x_0) \right) \ls 1
\]
holds. Combined with the lower estimate obtained in Proposition~\ref{prop:vp-const}, this yields \eqref{eq:condition-2}.
\end{rem}

\begin{rem}{\cite[Proposition~2.6]{Hatano-Ikeda-Ishikawa-Sawano2021}}\label{rem:Lipschitz}
Let $\vp \in \cGdec_1$ and $p \ge 1$. Assume that $r^{-n/p} \ls \vp(r)$ for $r>0$.
If $\vp(\alpha_1(x_0)) \ls 1$, then the inverse function $\psi^{-1}$ is Lipschitz.
Indeed, let $x,\tilde{x}\in\mathbb R^n$. Then
\begin{align*}
\left| \psi^{-1}(x)-\psi^{-1}(\tilde{x}) \right|
&\le
\sup_{z\in [x,\tilde{x}]}
\left\|D\psi^{-1}(z)\right\|_{\rm Fro}
|x-\tilde{x}|
\\
&\le
\sup_{z\in [x,\tilde{x}]}
\left(
\sum_{i=1}^n
\frac1{\alpha_i(\psi^{-1}(z))^2}
\right)^{1/2}
|x-\tilde{x}|
\\
&\ls
\sup_{z\in [x,\tilde{x}]}
\left(
n\vp(\alpha_1(\psi^{-1}(z)))^{2p/n}
\right)^{1/2}
|x-\tilde{x}|
\ls
|x-\tilde{x}|,
\end{align*}
where we use the singular value decomposition for the second inequality.

\end{rem}

To prove Theorem~\ref{thm:OM-OM-second-necessary},  
we compute the Orlicz--Morrey norm of  $\chi_{\prod_{j=0}^{k-1}[0,a_j]\times\R^{n-k}}$ for $a_j>0$ and $k \in [1,n]\cap \N$,  whose definition is given by Definition~\ref{defn:OrliczMorrey}. 
However, it is hard to compute the integral of the characteristic function on the ball. 
Then we introduce another definition of the Orlicz--Morrey norm using cubes. As we will show, these definitions are equivalent.
Let $b = (b_1, \dots, b_n)\in \R^n$ and $r>0$. Then 
\[
Q(b, r) \equiv\left\{ x = (x_1, \dots, x_n)\in \R^n: \max_{i=1,\dots,n} |x_i-b_i| \le r \right\}.
\]
We denote 
by 
$\ell(Q(b,r))$ the side-length of $Q(b,r)$. 

\begin{defn}
For $\Phi \in \iPy$,
and a cube $Q(b,r)$ with $b \in \R^n$ and $r>0$, let
\begin{align*}
\|f\|_{\Phi,Q(b,r)}
&\equiv
\inf\left\{ \lambda>0:
\frac{1}{|Q(b,r)|}
\int_{Q(b,r)} \!\Phi\!\left(\frac{|f(x)|}{\lambda}\right) dx \le 1
\right\}.
\end{align*}
For $\vp \in \cGdec_1$, let ${\cM}_{\Phi, {\rm cube}}^\vp(\R^n)$ be the set of all functions $f \in L^0(\R^n)$ such that the
following functional is finite:
\begin{align*}
\|f\|_{{\cM}_{\Phi, {\rm cube}}^\vp}
&\equiv
\sup_{b \in \R^n,\, r>0} \frac1{\vp(\ell(Q(b,r)))} \|f\|_{\Phi,Q(b,r)}.
\end{align*}
\end{defn}
There are equivalent norms when we replace cubes with balls.

We calculate ${\cM}_{\Phi, {\rm cube}}^\vp(\R^n)$-norm of  $\chi_{\prod_{j=0}^{k-1}[0,a_j]\times\R^{n-k}}$ below. 
\begin{lem}\label{lem:chi-est}
Let $\Phi \in \iPy$, $\vp \in \cGdec_1$, and $k \in [1,n]\cap\N$. Then,
\begin{equation}\label{eq:chi-est}
\|\chi_{\prod_{j=0}^{k-1}[0,a_j]\times\R^{n-k}}\|_{{\cM}_{\Phi, {\rm cube}}^\vp} =\sup_{R>0}\frac1{\vp(R)} \Phi^{-1}\left(\frac{R^k}{\prod_{j=0}^{k-1}\min(a_j,R) }\right)^{-1}.
\end{equation}
\end{lem}
\begin{proof}
Let $Q(b,R/2)$ with $b =(b_0, \dots, b_{n-1}) \in \R^n$ and $R>0$.
We put
\[
I_{Q(b, R/2)}(\lambda) \equiv
\frac1{|Q(b,R/2)|}
\int_{ x \in Q(b,R/2) } \Phi\left( \frac{\chi_{\prod_{j=0}^{k-1}[0,a_j]\times\R^{n-k} }(x) }\lambda \right)dx
\]
for any $\lambda >0$. Then we have
\begin{align*}
&
I_{Q(b, R/2)}(\lambda)
\\&=
\frac1{|Q(b,R/2)|}
\int_{x_0=b_0-R/2}^{b_0+R/2} \cdots \int_{x_{n-1} = b_{n-1}-R/2}^{b_{n-1}+R/2}
\Phi\left(\frac1\lambda\right)\chi_{[0,a_0]}(x_0)\cdots \chi_{[0,a_{k-1}]}(x_{k-1})
\\&\hspace{7.0cm}\times\chi_{\R^{n-k}}(x_k, \dots,x_{n-1})dx_0\dots dx_{n-1}
\\&=
\frac1{R^k}
\Phi\left(\frac1\lambda\right)
\prod_{i=0}^{k-1}\int_{x_i=b_i-R/2}^{b_i+R/2} \chi_{[0,a_i]}(x_i)dx_i,
\end{align*}
where $x = (x_0, \dots, x_{n-1}) \in \R^n$, and we use $|Q(b,R/2)| = R^n$ and
\[
\int_{x_k=b_k-R/2}^{b_k+R/2}\cdots \int_{x_{n-1} = b_{n-1}-R/2}^{b_{n-1}+R/2}
\chi_{\R^{n-k}}(x_k, \dots,x_{n-1}) dx_k\dots dx_{n-1}= R^{n-k}
\]
for the last equality.
Here we see that
\begin{align*}
\sup_{b_i\in\R}\int_{x_i=b_i-R/2}^{b_i+R/2}\chi_{[0,a_i]}(x_i)dx_i
&=\min(a_i,R)
\end{align*}
for $i=0, \dots, k-1$. Therefore, we calculate
\begin{align}
&
\|\chi_{\prod_{j=0}^{k-1}[0,a_j]\times\R^{n-k}}\|_{{\cM}_{\Phi, {\rm cube}}^\vp}
\notag\\&=
\sup_{b \in \R^n,\, R>0} \frac1{\vp(R)} \inf\left\{ \lambda>0:
I_{Q(b, R/2)}(\lambda)\le 1
\right\}
\notag\\&=
\sup_{b \in \R^n,\, R>0} \frac1{\vp(R)}\Phi^{-1}\left(\frac{R^k}{ \prod_{i=0}^{k-1}\int_{x_i=b_i-R/2}^{b_i+R/2} \chi_{[0,a_i]}(x_i)dx_i }\right)^{-1}. 
\label{eq:Morrey_norm_in_chi2}
\end{align}
By the decreasingness of $t\mapsto\Phi^{-1}(t)^{-1}$ and taking the supremum on $b \in\R^n$  in \eqref{eq:Morrey_norm_in_chi2}, we have \eqref{eq:chi-est}.
\end{proof}

\section{Proof of Theorem~\ref{thm:OM-OM-second-necessary}}

\subsection{Case $n=1$}
In this section, we give a proof of (i) of Theorem~\ref{thm:OM-OM-second-necessary}.
\begin{proof}[ Proof of Theorem~\ref{thm:OM-OM-second-necessary}~{\rm (i)}]
Let $\alpha_1(x_0)>0$ with $x_0\in \R$ given in \eqref{eq:diag_a}. By Proposition~\ref{prop:vp-const} with $n=1$ and Remark~\ref{rem:vp-const}, we have $\vp(\alpha_1(x_0))\ls 1$, 
which implies that $\psi^{-1}$ is Lipschitz.
Applying the same argument to $\psi^{-1}$, using the boundedness
of $C_{\psi^{-1}}$, we also obtain that $\psi$ is Lipschitz.
Therefore, we have the conclusion.
\end{proof}

\subsection{Higher dimensional case $n\ge 2$}
In this section, we give a proof of Theorem~\ref{thm:OM-OM-second-necessary} when $n\ge 2$.
\subsubsection{An auxiliary function}

Hereafter, we define an auxiliary function,
\[
\Psi_{k,C}(r) \equiv \frac1{\vp(r)\Phi^{-1}(Cr^k)}
\]
for any $k \in [1,n]\cap\N$ and any $C>0$.

Then the following properties hold.

\begin{lem}\label{lem:func-dec-inc}
Let $C$ be  an arbitrary positive constant. Then
\begin{enumerate}\label{lem:inc_dec}
\item
Let $k \in [1,n-1]\cap\N$. Assume that for every $A>0$, $\Psi_{k, A}$ is  almost decreasing with a constant independent of $A$. Then $\Psi_{k+1, C}$ is almost decreasing.
\item
Let $k \in [2,n]\cap\N$. 
Assume that for every $A>0$,  $\Psi_{k,A}$ is almost increasing with a constant independent of $A$. Then $\Psi_{k-1, C}$ is almost increasing.
\end{enumerate}
\end{lem}
\begin{proof}
We only show Lemma~\ref{lem:func-dec-inc}~(i), since Lemma~\ref{lem:func-dec-inc}~(ii) 
is proved similarly. Let $r, s >0$ with $r<s$. Then
\[
\Psi_{k+1,C}(s) =\frac1{\vp(s)\Phi^{-1}(Cs^{k+1})}.
\]
Since $Cs^{k+1}=(Cs)s^k$, we have
\begin{align*}
\Psi_{k+1,C}(s)
&=
\frac1{\vp(s)\Phi^{-1}((Cs)s^k)}
=
\Psi_{k,Cs}(s).
\end{align*}
By the assumption, $\Psi_{k,Cs}$ is almost decreasing with
a constant independent of $Cs$. Hence,
\begin{equation}\label{eq:lem42-1}
\Psi_{k,Cs}(s) \le  C_1 \Psi_{k, Cs}(r) = C_1 \frac1{\vp(r)\Phi^{-1}((Cs) r^k)}.
\end{equation}
Since $r<s$, $Cr^{k+1}\le (Cs)r^k$. Because $\Phi^{-1}$ is increasing, $\Phi^{-1}(Cr^{k+1})\le \Phi^{-1}((Cs)r^k)$. Therefore
\begin{equation}\label{eq:lem42-2}
\frac1{\Phi^{-1}((Cs)r^k)} \le \frac1{\Phi^{-1}(Cr^{k+1})}.
\end{equation}
Combining \eqref{eq:lem42-1} and \eqref{eq:lem42-2}, we obtain
\[
\Psi_{k+1,C}(s) \le C_1 \frac1{\vp(r)\Phi^{-1}(Cr^{k+1})} = C_1 \Psi_{k+1,C}(r).
\]
Thus, $\Psi_{k+1,C}$ is almost decreasing.
\end{proof}


\subsubsection{Case $k=1$} 
In this section, we prove the case $k=1$ in Theorem~\ref{thm:OM-OM-second-necessary}.
\begin{lem}\label{lem:chi-Rn}
Let $n \ge 2$, $a_0>0$, $\Phi \in \iPy$, and $\vp \in \cGdec_1$. Assume that $\Psi_{1, C}$ is an almost decreasing function for any $C>0$. Then $\chi_{{[0,a_0]\times\R^{n-1}}}\in {\cM}_{\Phi, {\rm cube}}^\vp(\R^n)$ and
\[
\|\chi_{[0,a_0]\times\R^{n-1}}\|_{{\cM}_{\Phi, {\rm cube}}^\vp}\sim \frac1{\vp(a_0)}.
\]
\end{lem}

\begin{proof}
Recall that
\[
\Psi_{1, C}(r) \equiv \frac1{\vp(r)\Phi^{-1}(Cr)}
\]
for any $C>0$.
By Lemmas~\ref{lem:chi-est} and \ref{lem:func-dec-inc}, it suffices to show that
if $\Psi_{k, C}$ is an almost decreasing function for all  $C>0$, then
\begin{equation}\label{eq:chi-a_0}
\sup_{R>0}\frac1{\vp(R)} \Phi^{-1}\left(\frac{R}{\min(a_0,R) }\right)^{-1}
\sim \frac1{\vp(a_0)}.
\end{equation}

We divide the proof into two cases $R\le a_0$ and $R \ge a_0$ as follows: 

{\bf Case(i)}($R\le a_0$)
The left-hand side of \eqref{eq:chi-a_0} is estimated as
\begin{align*}
\frac1{\vp(R)}
\Phi^{-1}\left(\frac{R}{\min(a_0,R)}\right)^{-1} = \frac1{\vp(R) \Phi^{-1}(1)} \ls C_1\frac1{\vp(a_0)},
\end{align*}
where we use \eqref{eq:almost_decreasing} in the second inequality.
\\
{\bf Case(ii)}($R\ge a_0$)
The left-hand side of \eqref{eq:chi-a_0} is estimated as
\begin{align*}
\frac1{\vp(R)}
\Phi^{-1}\left(\frac{R}{\min(a_0,R)}\right)^{-1}
&=
\frac1{\vp(R) }
\Phi^{-1}\left(\frac{R}{a_0}\right)^{-1}
\ls C_1\frac1{\vp(a_0)},
\end{align*}
where we use the decreasingness of $\Psi_{1,1/a_0}(r)$ for $a_0>0$ in the second inequality. Therefore, we obtain \eqref{eq:chi-a_0} and the desired result.
\end{proof}

\begin{lem}\label{lem:operator_norm=infty}
Let $\Phi \in \iPy$, $\vp \in \cG_0\cap\cGdec_1$, $x_0\in \R^n$ and $\alpha_i\equiv \alpha_i(x_0) >0$ be positive numbers such that $\alpha_i \le \alpha_j$ for $1 \le i < j$.
If $\Psi_{1, C}$ is an almost decreasing function for any $C>0$, then 
\[{\displaystyle \liminf_{\alpha_1 \to 0}} \|C_{ {\rm diag}(\alpha_1, \cdots, \alpha_n)}\|_{
{\cM}_{\Phi, {\rm cube}}^\vp
\to 
{\cM}_{\Phi, {\rm cube}}^\vp
} = \infty.\]
\end{lem}

\begin{proof}
Put $D = {\rm diag}(\alpha_1, \cdots, \alpha_n)$ and $M = \|C_D\|_{{\cM}_{\Phi, {\rm cube}}^\vp\to {\cM}_{\Phi, {\rm cube}}^\vp}$. By Lemma~\ref{lem:chi-Rn}, we estimate
\begin{align}\label{eq:bdd-M2}
\frac1{\vp(\alpha_1^{-1})} \sim\left\|\chi_{[0,\alpha_1^{-1}]\times\R^{n-1}}\right\|_{
{\cM}_{\Phi, {\rm cube}}^\vp
}
&=
\left\|\chi_{[0,1]\times\R^{n-1}} \circ D\right\|_{
{\cM}_{\Phi, {\rm cube}}^\vp
}
\notag\\&\le M \left\|\chi_{[0,1]\times\R^{n-1}} \right\|_{
{\cM}_{\Phi, {\rm cube}}^\vp}
 \sim M\frac1{\vp(1)}.
\end{align}
By \eqref{eq:bdd-M2}, we have
\begin{equation}\label{eq:a1->0}
\liminf_{\alpha_1\to 0} M \gs \liminf_{\alpha_1\to 0} \frac{\vp(1)}{\vp(\alpha_1^{-1})} = \infty,
\end{equation}
where we use $\vp \in {\mathcal G}_0$ for the second equation, which yields $\varphi(1)\neq 0$ and $\varphi (\alpha_1^{-1})\to 0$ as $\alpha _1^{-1} \to \infty$, and hence  $1/\varphi (\alpha_1^{-1})\to \infty$.
Therefore, we have the conclusion.
\end{proof}
If $\vp \in\cGdec_2$, then \eqref{eq:bdd-M2} gives
\[
M  \gs  \frac{\vp(1)}{\vp(\alpha_1^{-1})} \gs \vp(\alpha_1).
\]
By the previous inequalities, we have the following corollary.

\begin{cor}\label{cor:vp<M}
We use the same notation as in Lemma~\ref{lem:operator_norm=infty}. If $\vp \in \cGdec_2$, then
$
\vp(\alpha_1)\ls  \|C_{ {\rm diag}(\alpha_1, \cdots, \alpha_n)}\|_{
{\cM}_{\Phi, {\rm cube}}^\vp
\to
{\cM}_{\Phi, {\rm cube}}^\vp
}.
$
\end{cor}

%

\subsubsection{Case $k \in [2, n]\cap\N$}

In this section, we show the case $k \in [2, n]\cap\N$ in Theorem~\ref{thm:OM-OM-second-necessary}.
\begin{lem}\label{lem:chi-norm-multi-second}
Let $n \ge 2$, $k \in [2, n]\cap\N$, $\Phi \in \iPy$, and $\vp \in \cGdec_1$. Assume that the
sequence $\{a_j\}_{j=0}^{n-1}$ satisfies $0 < a_0 \le \dots \le a_{n-1}$.
Assume that $\Psi_{k-1,C}$ is almost increasing and $\Psi_{k,C}$ is almost decreasing,
uniformly in $C>0$.
Then $\chi_{\prod_{j=0}^{n-1}[0,a_j]}\in {\cM}_{\Phi, {\rm cube}}^\vp(\R^n)$, and
\begin{equation}\label{eq:chi-norm-multi-second}
\left\|\chi_{\prod_{j=0}^{n-1}[0,a_j]}\right\|_{{\cM}_{\Phi, {\rm cube}}^\vp} \sim
\frac1{\vp(a_{k-1})}\Phi^{-1}
\left(\frac{a_{k-1}^{k-1}}{\prod_{j=0}^{k-2}a_j}\right)^{-1}.
\end{equation}
\end{lem}
\begin{proof}
Since we have only to consider cubes of the form $[0,R]^n$ for $R > 0$ as candidates for the supremum in the Orlicz--Morrey norm $\|\cdot\|_{{\cM}_{\Phi, {\rm cube}}^\vp}$, we have the identity
\begin{equation}\label{eq:chi-a_n}
\|\chi_{\prod_{j=0}^{n-1}[0,a_j]}\|_{{\cM}_{\Phi, {\rm cube}}^\vp} =\sup_{R>0}\frac1{\vp(R)} \Phi^{-1}\left(\frac{R^n}{\prod_{j=0}^{n-1}\min(a_j,R) }\right)^{-1} \equiv \sup_{R>0} F(R).
\end{equation}
We first evaluate $F$ at $R=a_{k-1}$. Then we estimate $F$ on $0<R\le a_{k-1}$ and $R\ge a_{k-1}$. 
Set
\begin{equation*}
T\equiv \frac1{\vp(a_{k-1})}\Phi^{-1}
\left(\frac{a_{k-1}^{k-1}}{A_{k-1}}\right)^{-1},
\end{equation*}
where $A_m =\prod_{j=0}^{m-1}a_j$.\\
{\bf Case (i)}$(R = a_{k-1})$
Evaluating $F$ at $R=a_{k-1}$, we have $F(a_{k-1}) =T$.
Therefore, we have
\[
\sup_{R>0} F(R) \ge T.
\]
{\bf Case (ii)}($0<R\le a_{k-1}$)
In this case, we have
\[
\prod_{j=0}^{n-1}\min(a_j,R)  = \prod_{j=0}^{k-2} \min(a_j,R) \prod_{j=k-1}^{n-1} \min(a_j,R)  \le A_{k-1}R^{n-k+1}.
\]
Hence
\[
\frac{R^n}{ \prod_{j=0}^{n-1}\min(a_j,R)   } \ge \frac{R^{k-1}}{A_{k-1}}.
\]
Since $\Phi^{-1}(\cdot)^{-1}$ is decreasing, we estimate
\[
F(R) \le \Psi_{k-1, A_{k-1}^{-1}}(R) \ls \Psi_{k-1, A_{k-1}^{-1}}(a_{k-1})  =T,
\]
where we use almost increasingness of $\Psi_{k-1,A_{k-1}^{-1}}$ for the second inequality.

{\bf Case (iii)}($R \ge a_{k-1}$)
In this case, we have
\[
\prod_{j=0}^{n-1}\min(a_j,R)  = \prod_{j=0}^{k-1} \min(a_j,R) \prod_{j=k}^{n-1} \min(a_j,R)  \le A_kR^{n-k}.
\]
Hence,
\[
\frac{R^n}{ \prod_{j=0}^{n-1}\min(a_j,R)   } \ge \frac{R^{k}}{A_{k}}.
\]
Since $\Phi^{-1}(\cdot)^{-1}$ is decreasing, we estimate
\begin{equation}\label{eq:increasing}
F(R) \le \Psi_{k, A_{k}^{-1}}(R) \ls \Psi_{k, A_{k}^{-1}}(a_{k-1})  =T,
\end{equation}
where we use almost decreasingness of $\Psi_{k,A_k^{-1}}$ for the second inequality. Here we used
$A_k=a_{k-1}A_{k-1}$ in the last equality.

Therefore,

\[
T
\le
\sup_{R>0}F(R)
\lesssim
T.
\]

Hence, by \eqref{eq:chi-a_n},

\[
\|\chi_{\prod_{j=0}^{n-1}[0,a_j]}\|_{{\cM}_{\Phi,{\rm cube}}^\vp}
\sim
\frac1{\vp(a_{k-1})}
\Phi^{-1}
\left(
\frac{a_{k-1}^{k-1}}{A_{k-1}}
\right)^{-1}.
\]
This completes the proof.
\end{proof}
Using the definition of operator norm~\eqref{eq:operator_norm} and Lemma~\ref{lem:chi-norm-multi-second}, we obtain the following corollary:
\begin{cor}\label{chi-operator-norm-multi}
Under the assumptions of Lemma~\ref{lem:chi-norm-multi-second}, the following holds.
\begin{align*}
&
\|C_{{\rm diag}(a_0, \dots, a_{n-1}) }\|_{{\cM}_{\Phi, {\rm cube}}^\vp \to{\cM}_{\Phi, {\rm cube}}^\vp   }
\notag\\&\gs
\frac{\vp(R_{k-1})}{ \vp(a_{k-1}^{-1}R_{k-1}) }
\Phi^{-1}\left(\frac{R_{k-1}^{k-1}}{\prod_{j=1}^{k-2}R_j}\right)
\Phi^{-1}\left(\frac{(a_{k-1}^{-1}R_{k-1})^{k-1} }{a_0^{-1}\prod_{j=1}^{k-2}a_j^{-1} R_j}\right)^{-1}
\end{align*}
for $1 \le R_1 \le \cdots \le R_{n-1}$ with $0< a_0^{-1} \le a_1^{-1}R_1 \le \cdots \le a_{n-1}^{-1}R_{n-1}$.
\end{cor}

\begin{proof}[ Proof of Theorem~\ref{thm:OM-OM-second-necessary}~{\rm (ii)}]
Let $x_0\in \R^n$ and $\alpha_j \equiv \alpha_j(x_0)$ for $j \in [1,n]\cap\N$. First, we show that $\vp(\alpha_1(x_0))\ls 1$. We consider the cases $k=1$ and $k\in [2,n]\cap\N$.
 
\noindent
{\bf Case(i)}($k=1$)
By Corollary~\ref{cor:vp<M}, we have $\vp(\alpha_1(x_0))\ls 1$.

\noindent
{\bf Case(ii)}($k\in [2,n]\cap\N$)
Set $\alpha_i \le \alpha_j$ if $1 \le i \le j \le n$, and $\Sigma(x_0) \equiv {\rm diag}(\alpha_1, \cdots, \alpha_n)$. Then, by Corollary~\ref{chi-operator-norm-multi}, we have
\begin{equation}\label{eq:diag_chi}
\|C_{\Sigma(x_0)}\|_{{\cM}_{\Phi,{\rm cube}}^\vp\to{\cM}_{\Phi,{\rm cube}}^\vp}
\gs
\frac{\vp(R_{k-1})}{\vp(\alpha_k^{-1}R_{k-1})}
\Phi^{-1}\left(
\frac{R_{k-1}^{k-1}}{\prod_{i=1}^{k-2}R_i}
\right)
\Phi^{-1}\left(
\frac{(\alpha_k^{-1}R_{k-1})^{k-1}}
{\alpha_1^{-1}\prod_{i=1}^{k-2}\alpha_{i+1}^{-1}R_i}
\right)^{-1}
\end{equation}
for $1 =R_0\le R_1 \le \cdots \le R_{n-1}$ with $0< \alpha_1^{-1} \le \alpha_2^{-1}R_1 \le \cdots \le \alpha_n^{-1}R_{n-1}$.
Since $\Psi_{k-1,C}$ is an almost increasing function and $\Phi^{-1}(\cdot)^{-1}$ is a  decreasing function, the right-hand side of \eqref{eq:diag_chi} is estimated as follows:
\begin{align}\label{eq:lower_est}
&\notag
\vp(R_{k-1})
\Phi^{-1}\left(\frac{R_{k-1}^{k-1}}{\prod_{i=1}^{k-2}R_i}\right)
\Phi^{-1}\left(\frac{(\alpha_k^{-1}R_{k-1})^{k-1} }{\alpha_1^{-1}\prod_{i=1}^{k-2}\alpha_{i+1}^{-1} R_i}\right)^{-1}
\\&\gs \cdots \gs
\vp(R_1)
\Phi^{-1}\left(\frac{R_1}{R_0}\right)
\Phi^{-1}\left(1\right)^{-1} \gs \vp(1).
\end{align}
By using $\vp \in \cGdec_2$ and \eqref{eq:lower_est}, we have
\begin{equation*}
1 \gs
\|C_{\Sigma(x_0) }\|_{{{\cM}_{\Phi, {\rm cube}}^\vp}\to{{\cM}_{\Phi, {\rm cube}}^\vp}}
\gs
\frac{\vp(1)}{\vp(\alpha_k^{-1}R_{k-1})}
\gs
\frac{\vp(1)}{ \vp(\alpha_1^{-1}) }
\gs \vp(\alpha_1).
\end{equation*}
By the results in Case~(i) and Case~(ii), we obtain $\vp(\alpha_1(x_0))\ls 1$.
The inequality $\vp(\alpha_1(x_0))\ls 1$ holds for every $k \in [1, n]\cap \N$, independent of the choice of $k$. Therefore, \eqref{eq:condition-2} is satisfied, and the inverse mapping $\psi^{-1}$ is Lipschitz. Applying the same argument to $\psi^{-1}$ and using the boundedness of both $C_\psi$ and $C_{\psi^{-1}}$, we also conclude that $\psi$ is Lipschitz. Using Proposition~\ref{prop:vp-const}, the desired conclusion follows.
\end{proof}

\section{Proof of Theorem~\ref{thm:wOM-wOM}}

We now define weak-type spaces and introduce composition operators in weak-type spaces.
\begin{defn}[{\rm Weak type spaces, see~\cite[Definition~1.12]{Hatano-Ikeda-Ishikawa-Sawano2021}}]\label{defn:weak-type composition}
Let $(B(\R^n), \|\cdot\|_B)$ be a linear subspace of $L^0(\R^n)$ such that $\| |f| \|_B = \| f \|_B$
for all $f \in B(\R^n)$. The weak-type space $(\mathrm{w}\hskip-0.6pt{B}(\R^n), \|\cdot\|_{\mathrm{w}\hskip-0.6pt{B}})$ of $B(\R^n)$ is defined by
\[
\mathrm{w}\hskip-0.6pt{B}(\R^n) \equiv \left\{f \in L^0(\R^n): \left\|f\right\|_{ \mathrm{w}\hskip-0.6pt{B} } < \infty\right\},
\]
endowed with the quasi-norm
\[
\left\|f\right\|_{ \mathrm{w}\hskip-0.6pt{B} }
\equiv \sup_{\lambda>0} \lambda \left\|\chi_{\left\{x\in \R^n: |f(x)| >\lambda\right\}}\right\|_B.
\]
\end{defn}

\begin{prop}{\rm \cite[Theorem~1.13]{Hatano-Ikeda-Ishikawa-Sawano2021}}\label{prop:weak-type composition}
Let $(B(\R^n), \|\cdot\|_B)$ be a quasi-normed space. Then, $C_\psi$ induced by $\psi$
is bounded on the weak-type space $(\mathrm{w}\hskip-0.6pt{B}(\R^n), \|\cdot\|_{\mathrm{w}\hskip-0.6pt{B}})$ if and only if there exists a constant $K>0$ such that for all measurable sets $A$ in $\R^n$, the estimate holds as follows.
\[
\left\|\chi_{\psi^{-1}(A)}\right\|_{B} \le K \left\|\chi_{A}\right\|_{B}.
\]
In particular, we obtain
\[
\left\|C_\psi\right\|_{\mathrm{w}\hskip-0.6pt{B} \to \mathrm{w}\hskip-0.6pt{B} } = \sup_A \frac{\left\|\chi_{\psi^{-1}(A)}\right\|_{B} }{ \left\|\chi_{A}\right\|_{B}},
\]
where the supremum is taken over all measurable sets $A\subset \R^n$ with $\left\|\chi_A\right\|_{B} \in (0,\infty)$.
\end{prop}
We will apply Proposition~\ref{prop:weak-type composition} with $B(\R^n) =\cM_\Phi^\vp(\R^n)$ to prove Theorem~\ref{thm:wOM-wOM} since the Orlicz--Morrey spaces are quasi-normed spaces. Hence, it remains to show that the weak Orlicz--Morrey spaces are weak-type spaces. In other words, we have the following assertion:
\begin{lem}\label{lem:weak-norm}
Let $\Phi \in \iPy$, and $\vp\in\cGdec_1$. Then
\[
\|f\|_{\wcM_\Phi^\vp} = \sup_{\lambda>0}\lambda\left\|\chi_{\{x\in \R^n :|f(x)| >\lambda\}} \right\|_{{\mathcal M}_\Phi^\vp}.
\]
\end{lem}

We can prove this lemma in the same way as~\cite[Lemma~2.4]{Hatano-Kawasumi-Ono2023} about the weak Orlicz spaces (see~\cite[Lemma~2.4]{Hatano-Kawasumi-Ono2023} for the definitions). Therefore, we omit the details.

%
%
%

\medskip

{\it Acknowledgement.} 
This work was supported by the Research Institute for Mathematical Sciences,
an International Joint Usage/Research Center located in Kyoto University. he third author presented this work at the RIMS Workshop entitled Research on Real, Complex and Functional Analysis from the Perspective of Reproducing Function Spaces. We appreciate the comments received from the organizers and participants at the RIMS
Workshop.\newline

{\it Funding.} The first and second authors acknowledge support from the Japan Science and Technology Agency (JST) under CREST grant JPMJCR1913. The second author acknowledges support from the JST under ACT-X grant JPMJAX2004.\newline

{\it Availability of data and material.} No data or materials were used to support this study.
\newline

{\it Competing interests.} The authors declare that there are no conflicts of interest
regarding the publication of this paper.\newline

{\it Authors' contributions.} The three authors contributed equally to this paper and  reviewed the manuscript.\newline

{\it AI disclosure.} In this study, we used OpenAI ChatGPT (GPT-5.5) to improve English expression and assist in reviewing mathematical statements and proofs. All AI-generated outputs were carefully reviewed and verified by the authors.


\section*{Appendix}
Here, we show~\eqref{eq:appendix_A} and \eqref{eq:appendix_B}.
\subsection*{Proof of Remark~\ref{rem:classical-morrey} (iii)}
Let $e \le q \le p$. We prove that
\begin{equation*}
\cM_2^\rho(\R^n)\cap L^\infty(\R^n) \subset \cM_\Phi^\vp(\R^n).
\end{equation*}
Assume that $f \in \cM_2^\rho(\R^n)$ and $f \in L^\infty(\R^n)$ with $\rho= 2p/q $. 
The case $f=0$ is trivial. Hence we may assume that
\[
M\equiv\|f\|_{L^\infty}>0,
\qquad
g\equiv \frac{f}{M}.
\]
Then
\[
|g(x)|\le1
\qquad\text{a.e.}
\]
Moreover,
\[
\|g\|_{\cM_2^\rho}
=\frac1M
\|f\|_{\cM_2^\rho}
<\infty.
\]
By the definition of $\Phi$, we have
\[
\Phi(t)\le
\max(t^2,t^q) \equiv \Psi(t) .
\]
Then
\[
\Psi(\ve t) \le \max(\ve^2,\ve^q)\max(t^2,t^q)  =\Psi(\ve)\Psi(t).
\]
Hence, 
\begin{align*}
\left\|g\right\|_{\cM_\Phi^\vp} 
&\le\sup_{a\in \R^n, r>0}r^{n/p} \inf\left\{ \lambda>0 :  \Psi\left(\frac1\lambda\right)\frac1{|B(a,r)|} \int_{B(a,r)}  \Psi(|g(x)|) dx \le 1\right\}
\\&=\sup_{a\in \R^n, r>0}r^{n/p} \Psi^{-1}\left( \frac{|B(a,r)|}{\int_{B(a,r)}  \Psi(|g(x)|) dx }  \right)^{-1},
\end{align*}
where
\[
\Psi^{-1}(t)=
\begin{cases}
t^{1/2}, & 0\le t<1,\\
t^{1/q}, & 1\le t .
\end{cases}
\]

By $0 < |g(x)| <1$, we compute
\begin{align*}
0 < |g(x)| <1
&\Longrightarrow
0 < \Psi(|g(x)|)<\Psi(1) =1
\\&\Longrightarrow
0 < \frac1{|B(a,r)|}\int_{B(a,r)}\Psi(|g(x)|)dx<1
\\&\Longrightarrow
1 <  \frac{|B(a,r)|}{\int_{B(a,r)}  \Psi(|g(x)|) dx }.
\end{align*}
Consequently,
\begin{align*}
\left\|g\right\|_{\cM_\Phi^\vp} &\le \sup_{a\in \R^n, r>0}r^{n/p} \left(\frac1{|B(a,r)|} \int_{B(a,r)}  |g(x)|^2 dx   \right)^{1/q}
\\&\sim\
\left(
\sup_{a\in \R^n, r>0} r^{qn/2p -n/2} \left(\int_{B(a,r)}  |g(x)|^2 dx   \right)^{1/2}
\right)^{2/q}
\sim
\left\|g\right\|_{\cM_2^\rho}^{2/q} < \infty.
\end{align*}

Finally, since $f=Mg$ and
\[
\Psi(Mt)
\le
\Psi(M)\Psi(t),
\]
we obtain
\[
\frac1{|B(a,r)|}
\int_{B(a,r)}
\Psi(|f(x)|)\,dx
\le
\Psi(M)
\frac1{|B(a,r)|}
\int_{B(a,r)}
\Psi(|g(x)|)\,dx.
\]
Hence,
\[
\|f\|_{\cM_\Phi^\vp}
\le
\frac1{\Psi^{-1}(1/\Psi(M))}
\,
\|g\|_{\cM_\Phi^\vp}
<\infty.
\]
Therefore,
$f\in\cM_\Phi^\vp(\R^n)$.

\subsection*{Proof of Remark~\ref{rem:Linf}}
We prove only that if $0 < \inf_{r>0} \vp(r) \le \sup_{r>0} \vp(r) < \infty$, then 
\begin{equation}\label{eq:appendix_B}
\Phi^{-1}(1) \inf_{r>0}\vp(r) \|f\|_{\cM_\Phi^\vp} \le \|f\|_{L^\infty} \le \Phi^{-1}(1) \sup_{r>0}\vp(r) \|f\|_{\cM_\Phi^\vp}
\end{equation}
for $f \in {\cM_\Phi^\vp}(\R^n)$. First we show the left inequality in \eqref{eq:appendix_B}. In fact, we have
\[
\frac1{|B(a,r)|}\int_{B(a,r)} \Phi\left(\frac{|f(x)| }{\Phi^{-1}(1)^{-1}  \|f\|_{L^\infty}  }\right) dx \le 1
\]
and 
\[
 \inf_{r>0}\vp(r)\Phi^{-1}(1) \|f\|_{\cM_\Phi^\vp} \le \|f\|_{L^\infty} . 
\]

Regarding the right inequality in \eqref{eq:appendix_B}, by Lebesgue's differentiation theorem, we estimate
\begin{align*}
&
\Phi^{-1}(1)^{-1}|f(x)| 
\\&\le 
\inf\left\{\lambda >0 :\Phi\left(\frac{|f(x)|}\lambda\right) \le 1\right\}
\\&=\inf\left\{\lambda >0 : \lim_{r\to 0} \frac1{|B(x,r)|}\int_{B(x,r)}\Phi\left(\frac{|f(y)|}\lambda\right)dy \le 1\right\}
\\&\le \sup_{r>0}\frac{\vp(r)}{\vp(r)} \inf\left\{\lambda >0 : \frac1{|B(x,r)|}\int_{B(x,r)}\Phi\left(\frac{|f(y)|}\lambda\right)dy \le 1\right\}
\\&\le \sup_{r>0} \vp(r)\|f\|_{\cM_\Phi^\vp}.
\end{align*}
Therefore, we have
\begin{equation}\label{eq:AppendixB-1}
\|f\|_{L^\infty} \le \Phi^{-1}(1) \sup_{r>0}\vp(r) \|f\|_{\cM_\Phi^\vp}.
\end{equation}

\begin{remark*}
We present another method for deriving equation \eqref{eq:AppendixB-1} below:
Using Lebesgue's differentiation theorem and the embedding $\cM_\Phi^\vp(\R^n) \subset\cM_1^\vp(\R^n)$, we estimate
\begin{align*}
|f(x)| &\le \lim_{r\to 0} \frac1{|Q(x,r)|} \int_{Q(x,r)} |f(y)| dy
\\&\le \sup_{r>0}\vp(r) \|f\|_{\cM_1^\vp}
\\&\le \Phi^{-1}(1)\sup_{r>0}\vp(r) \|f\|_{\cM_\Phi^\vp}.
\end{align*}
Here, we note that the embedding $\cM_\Phi^\varphi(\R^n) \subset \cM_1^\varphi(\R^n)$ follows from equation (2.5) in \cite{Sawano-Sugano-Tanaka2011}: by setting $g=1$ in (2.5) and multiplying both sides by $\vp(\cdot)$, the desired result is obtained.
\end{remark*}

\end{document}